\documentclass[10pt,a4paper]{amsart}
\usepackage{bm}
\usepackage{amssymb}
\usepackage[final]{showkeys}
\usepackage{microtype}
\usepackage{color}
\usepackage{graphicx}
\usepackage{mathabx}
\usepackage[hmargin=3cm,vmargin={5cm,4cm}]{geometry}
\usepackage{dsfont}
\newtheorem{te}{Theorem}[section]

\newtheorem{os}[te]{Remark}
\newtheorem{prop}[te]{Proposition}

\newcommand{\de}{\mathrm{d}}
\numberwithin{equation}{section}
\newtheorem{theorem}{Theorem}
\newtheorem{proposition}{Propostion}

\newtheorem{remark}{Remark}[section]
\newtheorem{corollary}[theorem]{Corollary}

\allowdisplaybreaks
\begin{document}

\title[Stochastic solutions to abstract telegraph-type equations]{Stochastic solutions to abstract telegraph-type equations involving fractional dynamics}
\author{Alessandro De Gregorio}
\address{Department of Statistical Sciences, “Sapienza” University of Rome, P.le Aldo Moro, 5, 00185 Rome, Italy}
\email{alessandro.degregorio@uniroma1.it, Corresponding author}
\author{Roberto Garra}
\address{Section of Mathematics, International Telematic University Uninettuno, Corso Vittorio Emanuele II, 39, 00186 Rome, Italy}
\email{roberto.garra@uninettunouniversity.net}

\begin{abstract}
 This paper investigates abstract integro-differential hyperbolic equations, focusing on the probabilistic representation of their solutions. Our analysis is based on fractional derivatives and non-local operators, which are powerful tools for modeling the anomalous behavior and non-Markovian dynamics observed in various phenomena.
 
We first analyze a time-fractional version of the abstract telegraph equation (involving the Caputo derivative), restricting our analysis to positive self-adjoint operators to leverage spectral theory, which includes key operators in applications, such as the fractional Laplace operator. We derive analytical representations for the solution and provide a stochastic solution to the telegraph-diffusion equation for a specific range of the fractional parameter $\alpha$, thereby generalizing existing results. We discuss particular cases involving the fractional Laplace and Bessel-Riesz operators.

Furthermore, we consider the abstract Euler-Poisson-Darboux (EPD) equation, characterized by a singular time coefficient. We demonstrate that the stochastic solution to this EPD equation can be represented in terms of the solution of the abstract wave equation. Crucially, we prove that the solution to the EPD equation admits a representation by means of the Erdelyi-Kober fractional integral. 

Finally, this work provides a comprehensive analysis of both time-fractional and singular-coefficient abstract telegraph-type equations, offering new analytical and stochastic representation formulas.

   
\end{abstract}

\maketitle

\section{Introduction}

Fractional derivatives (i.e., derivative with non-integer order) and non-local operators are now recognized as powerful tools for analyzing the anomalous behavior observed in various phenomena. Such anomalies include nonlinear mean-squared displacement in time, heavy-tailed and skewed marginal distributions, and sample paths with jumps (see, e.g., \cite{metzler2000random}). In this context, time-changed random processes serve as valuable models for these non-Markovian fractional dynamics, explicitly incorporating the system's memory.

For example, replacing the classical time derivative in the heat equation with a time-fractional derivative yields a new integro-differential equation. The stochastic solution to this equation is known to be a Brownian motion subordinated by a random time (specifically, the inverse of a stable subordinator). For a comprehensive discussion on this subject, the reader may consult the overview \cite{meerschaert2013inverse}. Furthermore, fractional calculus is relevant in the analysis of several partial differential equations (PDEs) arising in mathematical physics. For instance, the Euler-Poisson-Darboux equation $$\left[\partial_{t}^2+\frac{2\lambda}{t}\partial_t\right]u=\Delta u,$$ which describes wave propagation, is well-known to admit a solution expressible as the Erdelyi-Kober fractional integral of the D'Alembert solution to the classical wave equation (see, e.g., \cite{erd1970}).

In this paper, we investigate abstract integro-differential hyperbolic equations. In particular we refer to two different families of generalized telegraph-type equations and study the probabilistic representation of their solutions. In our analysis the fractional calculus plays a two-fold role. Indeed, we study  a time-fractional version of the telegraph equation involving the fractional Caputo derivative. Analytic and stochastic solutions to fractional telegraph equations have been studied by several researchers (see, e.g., \cite{cascaval2002fractional}, \cite{orsingher2004time}, \cite{chen2008analytical}, \cite{d2014time}, \cite{li2019fractional} and \cite{angelani2024anomalous}). Furthermore, we introduce an abstract version of the EPD equation and prove that its solution admits a stochastic representation in terms of fractional integrals.

In order to introduce the main object of our interest, let us consider the abstract telegraph equation 
\begin{equation}\label{eq:atpde}
\left[\partial_{t}^2+2\lambda\partial_t+A\right]u=0,
\end{equation}
studied, e.g., in \cite{eckstein1999mathematics} and \cite{griego1971}, where $A$ is a linear operator. Let $\mathbb H$ be an Hilbert space,  with $||\cdot||$ representing the norm induced by the inner product. A time-fractional extension of \eqref{eq:atpde} is given by

\begin{align}\label{tee}
\left[(D_t^{\alpha})^2 +2\lambda D_t^{\alpha} +A\right] u=0, \quad t>0,\alpha\in(0,1],
\end{align}
where $(D_t^{\alpha})^2:=D_t^{\alpha}D_t^{\alpha}, u:=u(t)$, $A$ is a  positive (i.e., nonnegative and injective) self-adjoint operator on $\mathbb H$ (for instance $A=-\Delta=-\sum_{k=1}^d\partial_{x_k}^2$ with $\mathbb H=L^2(\mathbb R^d)$)
and $\lambda> 0.$ 
In \eqref{tee} we deal with time-fractional  Caputo (or Dzhrbashyan-Caputo) derivatives  as follows: let $f:\mathbb R^+_0\to \mathbb H$
\begin{equation}
D_t^\alpha f(t):=\frac{1}{\Gamma(1-\alpha)}\int_0^t \frac{f'(\tau)}{(t-\tau)^{\alpha}}\de\tau,\quad t>0, \ \alpha \in(0,1)
\end{equation}
where $f$ is a.e. differentiable (in norm) for $t>0$ and absolutely continuous, with the integral understood in Bochner's sense. Furthermore, we recall that the Laplace transform of the fractional derivative of order $\alpha \in (0,1)$ is given by
$$\widetilde{\left(D_t^\alpha f\right)}(s):=\int_0^\infty e^{-st}D_t^\alpha f(t)\,\de t=s^{\alpha}\widetilde{f}(s)-s^{\alpha-1}f(0),$$
where $\widetilde{f}(s)$ is the Laplace transform of $f$ and $s$ is a complex number with Re$(s)$ sufficiently large. For more details on fractional calculus, the reader can consult \cite{kilbas2006} and \cite{podlubny1998fractional}. It is worth mentioning that we decided to work with the less general $A$ because, in this setting, it is possible to apply the spectral theory of self-adjoint operators. Furthermore, the most interesting operators in applications fall into the class defined by $A$ (for instance, the fractional Laplace operator). In this paper, we derive analytical representations of the solution to the fractional telegraph equation \eqref{tee}. Furthermore, for $\alpha\in(0,\frac12)$, we are able to provide a stochastic solution to the telegraph-diffusion equation \eqref{tee}, which generalizes the result discussed in \cite{d2014time}. We discuss some particular cases involving operators of interest in applications, such as the fractional Laplace and Bessel-Riesz operators. The case $\alpha\in(\frac12,1)$ has been analyzed in \cite{li2019fractional}, for a general operator $A,$ where the solution is expressed in terms of d'Alembert formulas. It was also anlyzed in \cite{angelani2024anomalous}, for $A=-\partial^2_x,$ where the authors proved that the solution of the fractional telegraph equation admits a random time-changed Kac's representation.

A different class of abstract telegraph-type equations is given by the abstract EPD equation:
\begin{equation}\label{eq:atpde2}
\left[\partial_{t}^2+\frac{2\lambda}{t}\partial_t+A\right]u=0,
\end{equation}
where the singular coefficient $\lambda(t)=\frac{2\lambda}{t}$ appears,  and $A$ is defined as in \eqref{tee}. The equation \eqref{eq:atpde2} and its generalizations have attracted the attention of several researchers over the years; we mention, for instance, \cite{bresters}, \cite{rosencrans1973diffusion}, \cite{shishkina2017general} and \cite{de2020random}. A modified version of the EPD equation has been recently analyzed in \cite{martinucci2025finite}. We are able to represent the stochastic solution to \eqref{eq:atpde2} in terms of the solution of the abstract wave equation $\left[\partial_{t}^2+A\right]u=0.$ Furthermore, it is possible to show that the solution to the EPD equation admits a representation by means of the Erdelyi-Kober fractional integral.

\section{Abstract time-fractional telegraph equations}

\subsection{Introduction and results on fractional telegraph equations}\label{sec:intrfte} Let us deal with the same setting discussed  in \cite{cascaval2002fractional}.  Let  $(\mathbb H,\langle\cdot,\cdot\rangle)$ be a Hilbert space.   Let $A$ be a positive (i.e., $\langle Ax,x\rangle\geq 0, x\in D(A)$) and injective self-adjoint operator on  $\mathbb H$. The multiplication operator version of the spectral theorem allows the following result: there exists a unitary operator $ U:\mathbb H\to L^2(\mathfrak X, \Sigma, \mu)$ and a $\Sigma$-measurable function $m:\mathfrak X\to (0,\infty) $ unique (modulo changes on sets of $\mu$-measure zero) and positive ($\mu$ a.e.) such that
\begin{equation}\label{eq:specdec} U\, A\, U^{-1} f=M_mf:=m f,
\end{equation}
for
$$f\in \text{Dom}( U\, A\, U^{-1} )=\{f\in L^2(\mathfrak X, \Sigma, \mu): mf\in L^2(\mathfrak X, \Sigma, \mu)\}. $$
An interesting application of the spectral theorem is the chance to express a Borel function $G:\mathbb R^+\to\mathbb C$ of $A$ as product of operators; that is
$$G(A)=U^{-1}M_{G(m)}U.$$

Let us deal with the abstract time-fractional telegraph equation 

\begin{align}
&\bigg[(D_t^{\alpha})^2+2\lambda D_t^{\alpha}+A\bigg] u =0,\quad t> 0, \ \alpha \in(0,1], \label{eq:fractelc}\\
&u(0)=f,\ \alpha \in(0,1],\label{eq:fractelc2}\\
& D_t^{\alpha}u(0)=0,\ \alpha \in(1/2,1],\label{eq:fractelc3}
\end{align}
where $t\mapsto u(t)\in \mathbb H$ and $f\in D(A).$
In \cite{cascaval2002fractional}, the authors proved that
\begin{equation}
u_\alpha(t)=U^{-1}\widehat{u}_\alpha(t)
\end{equation}
is the unique solution of \eqref{eq:fractelc}-\eqref{eq:fractelc3}, where $\widehat u_\alpha(t,\xi)\in L^2(\mathfrak X, \Sigma, \mu), \xi\in \mathfrak X,$ is the solution to the following problem
\begin{align}
&\bigg[(D_t^{\alpha})^2+2\lambda D_t^{\alpha}+m(\xi)\bigg]\widehat u(t,\xi)=0,\label{1t}\\
&\widehat u(0,\xi)=\widehat f(\xi),\quad  D_t^{\alpha}\widehat u(0,\xi)=0.\label{1t2}
\end{align}

Furthermore, the Laplace transform of the solution to the problem \eqref{1t}-\eqref{1t2} is given by 
 \begin{align}\label{eq:laplacesol}
 \widetilde{\widehat{u}}_\alpha(s,\xi)&= \frac{(s^{2\alpha-1}+2\lambda
  s^{\alpha-1})\widehat f(\xi)}{s^{2\alpha}+2\lambda s^{\alpha} + m(\xi)},
 \end{align}
and its inverse allows to write down
\begin{align}\label{sol}
 \widehat{u}_\alpha(t,\xi) &= \widehat f(\xi)\widehat \Phi_\alpha(t,\xi) 
\end{align}
where
\begin{align}\label{eq:cf}
\widehat\Phi_\alpha(t,\xi) &:=\frac{1}{2}\left[
\bigg(1+\frac{\lambda}{\sqrt{\lambda-m(\xi)}}\bigg)E_{\alpha,1}(r_1t^{\alpha})+\bigg(1-\frac{\lambda}{\sqrt{\lambda-m(\xi)}}\bigg)E_{\alpha,1}(r_2t^{\alpha})\right]\\
&=E_{\alpha,1}(r_1t^{\alpha})+\frac{(2\lambda+r_2)t^\alpha}{r_1-r_2}[r_1E_{\alpha,\alpha+1}(r_1t^{\alpha})-r_2E_{\alpha,\alpha+1}(r_2t^{\alpha})]\notag
\end{align}
and 
\begin{align}
\nonumber & r_1 := -\lambda+\sqrt{\lambda^2-m(\xi)},\\
\nonumber & r_2 := -\lambda-\sqrt{\lambda^2-m(\xi)},
\end{align}
$$E_{\alpha,\beta}(x) := \sum_{k=0}^\infty \frac{x^{ k}}{\Gamma(\alpha k+\beta)},\quad \alpha,\beta>0, x\in\mathbb C.$$
This latter is the well-known Mittag-Leffler function and \eqref{eq:cf} follows taking into account that $E_{\alpha,\alpha+1}(x)=\frac1x( E_{\alpha,1}-1)$ (see e.g. \cite{gorenflo2020mittag} for details).  We refer to \cite{cascaval2002fractional} for the detailed calculations and the asymptotics of $u_\alpha$. For a discussion on the strong solution of the problem \eqref{eq:fractelc}-\eqref{eq:fractelc3}, the reader can consult  \cite{ashurov2023fractional}. 


\begin{remark}\label{rem1}
    We highlight that $(D_t^\alpha)^2\neq D_t^{2\alpha}.$ Nevertheless, the Cauchy problem \eqref{eq:fractelc}-\eqref{eq:fractelc3} coincides with the following problem involving another kind of fractional telegraph equation (studied, for instance, in \cite{orsingher2004time} and \cite{d2014time})
    \begin{align}\label{eq:fractel2alpha}
&\bigg[D_t^{2\alpha}+2\lambda D_t^{\alpha}+A\bigg] u =0,\quad t> 0, \ \alpha \in(0,1], \\
&u(0)=f,\ \alpha \in(0,1],\notag\\
& \partial_tu(0)=0,\ \alpha \in(1/2,1].\notag
\end{align}
Indeed, under suitable assumptions on $u$, we recall that the Laplace transform for $D_t^{2\alpha}$ is equal to (see Lemma 2.24 in \cite{kilbas2006theory})
$$\widetilde{(D_t^{2\alpha} u)}(s)=s^{2\alpha}\widetilde{u}-s^{2\alpha-1}u(0)-s^{2\alpha-2}\partial_tu(0).$$
Therefore, by using the spectral theorem as above, from \eqref{eq:fractel2alpha}, we get 

\begin{align*}
&\bigg[(D_t^{2\alpha})+2\lambda D_t^{\alpha}+m\bigg]\widehat u(t)=0,\\
&\widehat u(0)=\widehat f,\quad  D_t^{\alpha}\widehat u(0)=0.
\end{align*}
By applying the Laplace transform to the previous problem, we derive that $\widetilde{\widehat u}$ coincides with \eqref{eq:laplacesol}.

\end{remark}
In many interesting cases the Hilbert space $\mathbb H$ coincides with $L^2(\mathbb R^d,\de x)$ or with a suitable subspace of $L^2(\mathbb R^d,\de x).$ 
Furthermore, it is well-known (by Plancharel theorem) that the Fourier transform of a function $f\in L^2(\mathbb R^d,\de x),$ that is $\widehat f(\xi):=\mathcal Ff(\xi):=\int_{\mathbb R^d} e^{i \langle x,\xi\rangle}f(x)\de x,$ with $\xi\in \mathbb R^d,$  represents a unitary operator of $L^2(\mathbb R^d,\de x)$ onto $L^2(\mathbb R^d,\de x)$.  Furthermore, we define inverse Fourier transform as $\mathcal F^{-1}f(x):=\frac{1}{(2\pi)^d}\int_{\mathbb R^d} e^{-i\langle x,\xi\rangle}\widehat f(\xi) \de \xi,$ where $x\in\mathbb R^d.$

\begin{theorem}\label{thm:ft}Let $A$ be a positive self-adjoint operator on $\mathbb H\subseteq L^2(\mathbb R^d,\de x),$ such that the following spectral decomposition holds
$$ A =\mathcal F^{-1}\,m\,\mathcal F$$ and $m$ is an isotropic function; that is $m(\xi)=m(||\xi||),$ for each  $\xi\in \mathbb R^d.$ The solution of the Cauchy problem \eqref{eq:fractelc}-\eqref{eq:fractelc3} becomes the convolution product $u_\alpha(t,x)=(f\ast \Phi_\alpha(t))(x),$ for each $x\in\mathbb R^d,$
where
\begin{equation}\label{eq:invfoursol}
\Phi_\alpha(t,x)=\frac{1}{(2\pi)^{\frac d2}||x||^{\frac d2-1}}\int_0^\infty r^{\frac d2}J_{\frac d2-1}(r||x||)\widehat\Phi_\alpha(t,r)\de r,\quad t>0.
\end{equation}
\end{theorem}
\begin{proof} In this setting, we have that $U=\mathcal F$ with $\mathcal F:\mathbb H\to L^2(\mathbb R^d,\de x)$ and $U^{-1}=\mathcal F^{-1},$ and $\widehat u_\alpha(t)$ in \eqref{sol} coincides with the Fourier transform of the solution of the Cauchy problem \eqref{eq:fractelc}-\eqref{eq:fractelc3}. Therefore, it is enough to compute $\Phi_\alpha$ as inverse Fourier transform of $\widehat\Phi_\alpha$. Let $x:=(x_1,...,x_d)\in\mathbb R^d,$ we can write down 
\begin{align*}
    \Phi_\alpha(t,x)&=(\mathcal F^{-1}\widehat\Phi_\alpha(t))(x)\\
    &=\frac{1}{(2\pi)^d}\int_{\mathbb R^d} e^{-i\langle x,\xi\rangle}\widehat\Phi_\alpha(t,||\xi||) \de \xi\\
    &=\text{(by spherical coordinates transformations)}\\
    &=\frac{1}{(2\pi)^d} \int_0^\infty r^{d-1}\widehat\Phi_\alpha(t,r)\de r\int_0^\pi\de\theta_1\cdots\int_0^\pi\de\theta_{d-1}\int_0^{2\pi}\de\phi\sin^{d-2}\theta_1\cdots\sin\theta_{d-2}\\
    &\quad \times \exp\{-ir(x_d(\sin\theta_1+...+\sin\theta_{d-2}\sin\phi)+...+x_2\sin\theta_1\cos\theta_2+x_1\cos\theta_1)\}
\end{align*}
and the result \eqref{eq:invfoursol} follows by noticing that (see, e.g., \cite{gradshteyn2014table})
\begin{align*}
&\int_0^\pi\de\theta_1\cdots\int_0^\pi\de\theta_{d-1}\int_0^{2\pi}\de\phi\sin^{d-2}\theta_1\cdots\sin\theta_{d-2}\\
    & \times \exp\{-ir(x_d(\sin\theta_1+...+\sin\theta_{d-2}\sin\phi)+...+x_2\sin\theta_1\cos\theta_2+x_1\cos\theta_1)\}\\
&=(2\pi)^{\frac d2}\frac{J_{\frac d2-1}(r||x||)}{(r||x||)^{\frac d2-1}}.
    \end{align*}
\end{proof}

\begin{remark}
It is worth mentioning that \eqref{eq:invfoursol} simplifies in some particular spaces. Since   $J_{-\frac12}(x)=\sqrt{\frac{2}{\pi x}}\cos(x),$ for $d=1$, we get that
$$\Phi_\alpha(t,x)=\frac{1}{\pi}\int_0^\infty \cos(r||x||)\widehat\Phi_\alpha(t,r)\de r ,$$
while, by recalling that $J_{\frac12}(x)=\sqrt{\frac{2}{\pi x}}\sin(x),$ if $d=3$
$$\Phi_\alpha(t,x)=\frac{1}{2\pi^2||x||}\int_0^\infty r\sin(r||x||)\widehat\Phi_\alpha(t,r)\de r.$$
\end{remark}


\subsection{Stochastic solutions for $\alpha\in(0,\frac12]:$ telegraph-diffusion case}
Let us consider the following abstract Cauchy problem involving heat-type equations
with $\alpha \in(0,1/2]$
\begin{align}
&\bigg[(D_t^{\alpha})^2+2\lambda D_t^{\alpha}+A\bigg] u =0,\quad t> 0, \label{eq:fractelcbis}\\
&u(0)=f\label{eq:fractelc2bis},
\end{align}
where $u:\mathbb R^+_0\to \mathbb  H$ and $f\in \mathbb H.$

Let us introduce two independent totally positively 
skewed stable random processes $\{H_1^{2\alpha}(t):t\geq 0\}$ and $\{H_2^\alpha(t) : t\geq 0\},$ where $\alpha\in(0,\frac12],$ having characteristic function given by $$
\mathbb E(e^{i\xi H_2^\alpha(t)})=e^{-\sigma t|\xi|^\alpha(1-i\,\text{sgn}(\xi)\tan\frac{\alpha\pi}{2})},\quad \sigma=\cos\frac{\alpha\pi}{2}.
$$
Now, let us define the inverse
$\{\mathcal{L}^\alpha(t): t\geq0\}$ the (non-negative) inverse process of the sum $H_1^{2\alpha}(t)+(2\lambda)^{1/\alpha}H_2^\alpha(t)$  as follows
\begin{equation}\label{ldef}
\mathcal{L}^\alpha(t) := \inf \bigg\{s\geq 0: \ H_1^{2\alpha}(s)+(2\lambda)^{1/\alpha}H_2^\alpha(s)\geq t\bigg\}, \quad t, \lambda >0. 
\end{equation} 
We denote by $\ell_\alpha(x,t),x\geq 0, t>0,$ the density function of $\mathcal{L}^\alpha(t).$
For a deep discussion on the properties of this process the reader can consult \cite{d2014time}. 
However, this is the inverse of the sum of stable subordinators and it preserves the main properties of the inverse of a stable subordinator, i.e. it is a non-Markovian process with non-stationary, non-independent increments and non-decreasing continuous a.s. sample paths (see e.g. \cite{meerschaert2019inverse} for details).

\begin{theorem} Let $\alpha\in(0,\frac12].$ The unique solution to the problem \eqref{eq:fractelcbis}-\eqref{eq:fractelc2bis}, has the following stochastic  representation
\begin{equation}\label{eq:soldiff}
    u_\alpha(t)=\mathbb E\left[v(\mathcal{L}^\alpha(t))\right].
\end{equation}
where $v(t)=T(t)f, t\geq 0,$ is the unique solution to the abstract Cauchy problem
\begin{align}\label{eq:cauchypar}
[\partial_t+A]v=0,\quad v(0)=f\in D(A),
\end{align}
with $\{T(t):t\geq 0\}$ representing the $ C_0$ contraction semigroup on $\mathbb H$ generated by $-A.$
\end{theorem}
\begin{proof}
By means of the spectral theorem, we can observe that $\widehat v(t,\xi)=\widehat f(\xi) e^{-m(\xi) t},$ where $\widehat f(\xi)\in L^2(\mathfrak X,\Sigma,\mu),$ solves $$[\partial_t+m(\xi)]\widehat v(t,\xi)=0,\quad \widehat v(0,\xi)=\widehat f(\xi),$$ and then  $v(t)=U^{-1} \widehat v(t)$ is the unique solution to the problem \eqref{eq:cauchypar}. Furthermore, by Lumer-Phillips Theorem (see, e.g., Theorem 3.3, pag.26, in \cite{goldstein2017semigroups}), it follows that $-A$ is the generator of  a $ C_0$ contraction semigroup   $\{T(t):t\geq 0\}.$ Therefore the Cauchy problem \eqref{eq:cauchypar} is well-posed; that is the resolvent set $\rho(-A)$ is non-empty and for each $f\in D(-A)$ there is a unique solution $u:\mathbb R^+\to D(-A)$ of \eqref{eq:cauchypar} in $C^1(\mathbb R^+,\mathbb H),$ given by $v(t)=T(t)f, t\geq0$ (see, e.g., Theorem 1.2, pag.83, in \cite{goldstein2017semigroups}).  We recall that the Laplace transform (with respect the variable $t$) of the density function $\ell_\alpha(x,t)$ of $\mathcal{L}^\alpha(t)$ is given by (see (4.14) in \cite{d2014time})
\begin{equation}
\tilde{\ell}_\alpha(x,s)= (s^{2\alpha-1}+2\lambda s^{\alpha-1})e^{-xs^{2\alpha}-2\lambda xs^\alpha},
\end{equation}
where we denoted with $s$ the Laplace parameter. Now, for each $\xi\in \mathfrak X$
\begin{align*}
\int_0^\infty e^{-st}\mathbb  E(e^{-m(\xi)\mathcal{L}^\alpha(t)})\de t&=\int_0^\infty e^{-st}\left(\int_0^\infty e^{-m(\xi)x}\ell_\alpha(x,t)\de x \right)\de t\\
&=(s^{2\alpha-1}+2\lambda s^{\alpha-1})\int_0^\infty e^{-x(m(\xi)+s^{2\alpha}+2\lambda s^\alpha)}\de x\\
&=\frac{s^{2\alpha-1}+2\lambda
  s^{\alpha-1}}{s^{2\alpha}+2\lambda s^{\alpha} + m(\xi)}.
\end{align*}
From \eqref{eq:laplacesol} and the uniqueness of Laplace transform, it follows that
$$\widehat u_\alpha(t,\xi)=\mathbb  E[\widehat v(\mathcal{L}^\alpha(t),\xi)]$$
and then we can conclude the proof by observing that $U^{-1}$ is a linear operator independent by $t$.
\end{proof}

From the previous theorem immediately follows the next result.

\begin{corollary}\label{cor1}
Let $-A$ be the infinitesimal generator on $L^2(\mathbb R^d,\de x)$ of a strong Markov process $\{X(t):t\geq 0\}$ independent of $\mathcal{L}^\alpha(t).$ Then  for $\alpha\in(0,\frac12],$ 
\begin{equation}
  u_\alpha(t,x)=\mathbb E_x\left[f(X(\mathcal{L}^\alpha(t) ))\right] ,\quad t>0,x\in\mathbb R^d. 
\end{equation}
\end{corollary}

\subsection{Special case $\alpha=1/2$}
 We recall that (see, e.g., \cite{gorenflo2020mittag})
    \begin{equation}
        E_{1/2,1}(t) = e^{t^2}(1+\text{erf}(t)),
    \end{equation}
where 
$$\text{erf}(t) =\frac{1}{\sqrt{\pi}} \int_0^t e^{-\tau^2}d\tau $$
is the error function. Therefore, in the case $\alpha = 1/2$, we can represent the solution for the abstract fractional diffusion equation 
\begin{equation}\label{eq:fteq12}
    \bigg[(D_t^{1/2})^2+2\lambda D_t^{1/2}+A\bigg] u =0
\end{equation}
as follows
\begin{equation}
\Phi_{\frac12}(t,x)=\frac{1}{(2\pi)^{\frac d2}||x||^{\frac d2-1}}\int_0^\infty r^{\frac d2}J_{\frac d2-1}(r||x||)\widehat\Phi_{\frac12}(t,r)\de r,\quad t>0.
\end{equation}
where 
\begin{align}
\widehat\Phi_{\frac12}(t,\xi) :=\frac{1}{2}\left[
\bigg(1+\frac{\lambda}{\sqrt{\lambda-m(\xi)}}\bigg) e^{r_1^2t}(1+\text{erf}(r_1\sqrt{t}))
+\bigg(1-\frac{\lambda}{\sqrt{\lambda-m(\xi)}}\bigg)e^{r_2^2t}(1+\text{erf}(r_2\sqrt{t}))\right],
\end{align}
and 
\begin{align}
\nonumber & r_1 := -\lambda+\sqrt{\lambda^2-m(\xi)},\\
\nonumber & r_2 := -\lambda-\sqrt{\lambda^2-m(\xi)}.
\end{align}
Therefore, as $t   \to 0,$
$$\widehat\Phi_{\frac12}(t,\xi)\sim e^{(2\lambda^2-m(\xi))t}+\cosh(t\sqrt{\lambda^2-m(\xi)})+\frac{\lambda}{\sqrt{\lambda-m(\xi)}}\sinh(t\sqrt{\lambda^2-m(\xi)})$$
This case can be particularly useful in order to obtain the asymptotic behavior of the solution as pointed out in \cite{cascaval2002fractional}. Observe, moreover, that in general $(D_t^{1/2})^2\neq D_t$.

Furthermore, we can obtain an alternative interesting representation of \eqref{eq:cf} for $\alpha=1/2.$
\begin{proposition}
Let $\{B(t):t\geq 0\}$ be the standard Brownian motion. We get that
\begin{align*}
\widehat\Phi_\frac{1}{2}(t,\xi) =\mathbb E \left[g(|B(t)|,\xi)\right],
\end{align*}
where 
$$g(t,\xi):=e^{-\lambda t}\left[\cosh(t\sqrt{\lambda^2-m(\xi)})+\frac{\lambda}{\sqrt{\lambda-m(\xi)}}\sinh(t\sqrt{\lambda^2-m(\xi)})\right].$$
\end{proposition}
\begin{proof}
We observe that (see (4.3) in \cite{orsingher2004time})
$$E_{\frac12,1}(x)=\frac{2}{\sqrt{\pi}}\int_0^\infty e^{-w^2-2xw}\de w.$$
and then
\begin{align*}
   E_{\frac12,1}(r_i\sqrt t)&= \frac{2}{\sqrt{\pi}}\int_0^\infty e^{-w^2+2r_i\sqrt tw}\de w\\
   &=(2 \sqrt t w=z)\\
   &=\frac{1}{\sqrt{\pi t}}\int_0^\infty e^{-\frac{z^2}{4t}+r_iz}\de z\\
   &=\mathbb E\left[e^{|B(t)|r_i}\right],\quad i=1,2,
\end{align*}
where $\frac{1}{\sqrt{\pi t}}e^{-\frac{z^2}{4t}}, z>0,$ represents the density function of $|B(t)|$ for any $t>0.$
Therefore
\begin{align*}
\widehat\Phi_\frac{1}{2}(t,\xi) &=\frac{1}{2}\left[
\bigg(1+\frac{\lambda}{\sqrt{\lambda-m(\xi)}}\bigg)\mathbb E(e^{|B(t)|r_1})+\bigg(1-\frac{\lambda}{\sqrt{\lambda-m(\xi)}}\bigg)\mathbb E(e^{|B(t)|r_2})\right]\\
&=\frac{1}{2}\left[
\bigg(1+\frac{\lambda}{\sqrt{\lambda-m(\xi)}}\bigg)\frac{1}{\sqrt{\pi t}}\int_0^\infty e^{-\frac{z^2}{4t}+r_1z}\de z+\bigg(1-\frac{\lambda}{\sqrt{\lambda-m(\xi)}}\bigg)\frac{1}{\sqrt{\pi t}}\int_0^\infty e^{-\frac{z^2}{4t}+r_2z}\de z\right]\\
&=\frac{1}{2\sqrt{\pi t}}\int_0^\infty e^{-\frac{z^2}{4t}-\lambda z}\left(e^{z\sqrt{\lambda^2-m(\xi)}}+e^{-z\sqrt{\lambda^2-m(\xi)}}\right)\de z\\
&\quad +\frac{1}{2\sqrt{\pi t}}\frac{\lambda}{\sqrt{\lambda-m(\xi)}}\int_0^\infty e^{-\frac{z^2}{4t}-\lambda z}\left(e^{z\sqrt{\lambda^2-m(\xi)}}-e^{-z\sqrt{\lambda^2-m(\xi)}}\right)\de z\\
&=\frac{1}{\sqrt{\pi t}}\int_0^\infty e^{-\frac{z^2}{4t}-\lambda z}\left(\cosh(z\sqrt{\lambda^2-m(\xi)})+\frac{\lambda}{\sqrt{\lambda-m(\xi)}}\sinh(z\sqrt{\lambda^2-m(\xi)})\right)\de z\\
&=\mathbb E \left(e^{-\lambda|B(t)|}\left[\cosh(|B(t)|\sqrt{\lambda^2-m(\xi)})+\frac{\lambda}{\sqrt{\lambda-m(\xi)}}\sinh(|B(t)|\sqrt{\lambda^2-m(\xi)})\right]\right).
\end{align*}
\end{proof}

In the setting of Theorem \ref{thm:ft}, by applying the inverse Fourier transform to $\widehat\Phi_{\frac12}$, we get  the following stochastic representation of the fundamental solution of \eqref{eq:fteq12}
\begin{equation}\label{eq:ss12}
  u_{\frac12}(t,x)=(\mathcal F^{-1}\widehat\Phi_{\frac12}(t))(x)=\mathbb E\left[(\mathcal F^{-1} g(|B(t)|))(x)\right],\quad t>0, x\in\mathbb R^d.  
\end{equation}
If $d=1$ and $A=-\partial_{x}^2,$ which implies $m(\xi)=|\xi|^2,$ it is possible to invert the Fourier transform $\widehat\Phi_{\frac12}(t,\xi)$ and then   from \eqref{eq:ss12} we derive  the result in Theorem 4.2 in \cite{orsingher2004time} (see also Remark \ref{rem1}); that is (see, e.g., \cite{masoliver1996finite})
\begin{align*}
    (\mathcal F^{-1} g(t))(x)&=\frac{e^{-\lambda t}}{2}\left[\lambda I_0(\lambda\sqrt{t^2-x^2})+\partial_tI_0(\lambda\sqrt{t^2-x^2})\right]1_{|x|<t}\\
    &\quad +\frac{e^{-\lambda t}}{2}\left[\delta(x-t)+\delta(x+t)\right],
\end{align*}
which represents the density function of the telegraph process $\{T(t): t\geq 0\},$ where $T(t)=V(0)\int_0^t (-1)^{N(s)}\de s,$ $V(0)=\pm 1$ with probability 1/2 and $\{N(t):t\geq 0\}$ is a homogeneous Poisson process with rate $\lambda>0,$ independent of $V(0).$ Therefore $u_{\frac12}(t,x)$ coincides with the density function of $T(|B(t)|);$ that is
\begin{align*}
 &u_{\frac12}(t,x)\\
& =\frac{1}{2\sqrt{\pi t}}\int_0^\infty   e^{-\frac{z^2}{4t}-\lambda z}\left\{\left[\lambda I_0(\lambda\sqrt{z^2-x^2})+\partial_zI_0(\lambda\sqrt{z^2-x^2})\right]1_{|x|<z} +\left[\delta(x-z)+\delta(x+z)\right]\right\}\de z.
\end{align*}

\section{Applications}

\subsection{The $d$-dimensional space-time fractional telegraph equation}\label{sec:fraclaplace}
The space-time fractional telegraph equation has been an object of many recent papers, we refer for example to \cite{d2014time} and \cite{tawfik2018analytical}. The classical Laplace operator is replaced with its fractional version denoted by $-(-\Delta)^{\beta/2},\beta\in(0,2]$. We recall that by means of the spectral theorem, we can define the fractional Laplace operator $-(-\Delta)^{\beta/2}$ as a Fourier multiplier with symbol $-||\xi||^{\beta},$ where $\xi\in\mathbb{R}^d $; that is
\begin{equation}\label{eq:deffl}
    (\mathcal F(-\Delta)^{\beta/2} f)(\xi)=||\xi||^{\beta}\hat f(\xi),
\end{equation}
where $f\in$ Dom$((-\Delta)^{\beta/2})=\{f\in L^2(\mathbb R^d,\de x):\int_{\mathbb R^d}(1+||\xi||^{\beta})|\hat f(\xi)|^2\de \xi<\infty\}$ (see, e.g., \cite{kwasnicki2017ten} for a detailed discussion on the fractional Laplace operator).  Furthermore, we recall that $\{S_d^{\beta}:t\geq 0\}$ is a $d$-dimensional isotropic stable process with characteristic function given by
$$\mathbb E\left[e^{i\langle \xi, S_d^\beta(t)\rangle}\right]=e^{-t ||\xi||^\beta},\quad \xi\in\mathbb R^d.$$ In view of the general Theorem 1, we have the following result.
\begin{prop}
The solution of the d-dimensional space-time fractional telegraph equation 
\begin{align*}
&\bigg[(D_t^{\alpha})^2+2\lambda D_t^{\alpha}+(-\Delta)^{\beta/2}\bigg] u = 0,\quad t> 0, \ \alpha\in(0,1], \beta \in(0,2], \\
&u(0)=f,\ \alpha \in(0,1],\\
& D_t^{\alpha}u(0)=0,\ \alpha \in(1/2,1],
\end{align*}
is given by the convolution product $u_\alpha(t,x)=(f\ast \Phi_\alpha(t))(x),$ for each $x\in\mathbb R^d,$where
\begin{equation*}
\Phi_\alpha(t,x)=\frac{1}{(2\pi)^{\frac d2}||x||^{\frac d2-1}}\int_0^\infty r^{\frac d2}J_{\frac d2-1}(r||x||)\widehat\Phi_\alpha(t,r)\de r,\quad t>0
\end{equation*}
and 
\begin{align*}
\widehat \Phi_\alpha(t,\xi) :=\frac{1}{2}\left[
\bigg(1+\frac{\lambda}{\sqrt{\lambda-||\xi||^{\beta}}}\bigg)E_{\alpha,1}(r_1t^{\alpha})+\bigg(1-\frac{\lambda}{\sqrt{\lambda-||\xi||^{\beta}}}\bigg)E_{\alpha,1}(r_2t^{\alpha})\right],
\end{align*}
where 
\begin{align}
\nonumber & r_1 := -\lambda+\sqrt{\lambda^2-||\xi||^{\beta}},\\
\nonumber & r_2 := -\lambda-\sqrt{\lambda^2-||\xi||^{\beta}}.
\end{align}
Furthermore, for $\alpha\in(0,\frac12],$  $$u_\alpha(t,x)=\mathbb E_x\left[f(S_d^\beta(\mathcal{L}^\alpha(t) ))\right] ,\quad t>0,x\in\mathbb R^d. $$
\end{prop}

\begin{proof}
As a matter of fact, this useful representation of the solution to the space-time fractional telegraph equation can be seen as a special case of Theorem 1. Indeed, from the well-known Fourier transform of the fractional Laplacian \eqref{eq:deffl} and by observing that $m(\xi)=||\xi||^{\beta}$, we can immediately derive the statement of the theorem. 

Since $\mathbb E(e^{i\langle \xi, S_d^\beta(t)\rangle})=e^{-t ||\xi||^\beta}$ is the characteristic function of the fundamental solution to the equation
$$[\partial_t+(-\Delta)^{\beta/2}]v=0,$$
from Corollary \ref{cor1} follows the last result stated in the theorem, which essentially coincides with Theorem 4.1 in \cite{d2014time} (see also Remark \ref{rem1}).
\end{proof}
\begin{remark}
  It is worth mentioning that for $\beta=2,$ the fractional Laplace operator coincides with the classical Laplace operator and $\{S_d^2(t):t\geq 0\}$ is the $d$-dimensional standard Brownian motion. If $\beta=\frac12,$ the process  $\{S_d^{1/2}(t):t\geq 0\}$ represents the Cauchy process having density function
  given by the Cauchy law or Poisson kernel
  $$p(t,x)=\frac{\Gamma(\frac{d+1}{2})t}{[\pi(t^2+||x||^2)]^{\frac{d+1}{2}}},\quad t>0,x\in \mathbb R^d.$$
\end{remark}

\subsection{The time fractional telegraph-type equation involving Bessel-Riesz operator}

The Bessel-Riesz process is a $d$-dimensional L\'evy process $\{Y_d^{\beta,\gamma}( t):t\geq 0\}$ with joint characteristic function
\begin{equation}
\widehat{v}(t,\xi) = \mathbb{E} \left[e^{i\langle\xi, Y_d^{\beta,\gamma}(t)\rangle}\right]= e^{-t ||\xi||^\beta(1+||\xi||^2)^{\gamma/2}},
\end{equation}
coinciding with the Fourier transform of the Green function of the equation 
\begin{equation}
\partial_t u = - (-\Delta)^{\beta/2}(I-\Delta)^{\gamma/2}u, \quad, \gamma\geq 0, \beta \in (0,2],
\end{equation}
where $ (-\Delta)^{\beta/2}$ and $(I-\Delta)^{\gamma/2}$ are the inverses of the Riesz and the Bessel potential respectively. Therefore, the operator $- (-\Delta)^{\beta/2}(I-\Delta)^{\gamma/2}$ can be defined as pseudo-differential operator in the space of Fourier transforms as done in the previous section; that is
$$(\mathcal F(-\Delta)^{\beta/2}(I-\Delta)^{\gamma/2} f)(\xi)=||\xi||^\beta(1+||\xi||^2)^{\gamma/2}\hat f(\xi),$$
where $f\in$ Dom$( (-\Delta)^{\beta/2}(I-\Delta)^{\gamma/2})=\{f\in L^2(\mathbb R^d,\de x):\int_{\mathbb R^d}||\xi||^\beta(1+||\xi||^2)^{\gamma/2}|\hat f(\xi)|^2\de \xi<\infty\}$. The Bessel-Riesz operator and the related process have been studied in detail in \cite{anh2004riesz}. In particular, it is possible to prove that $Y_d^{\beta,\gamma}(t)=B_d(L_{\beta,\gamma}(t)),$ where $\{B_d(t):t\geq 0\}$ is a $d$-dimensional Brownian motion and $\{L_{\beta,\gamma}(t):t\geq 0\}$ is a Bessel-Riesz Lévy subordinator with Laplace transform $\exp\{-ts^{\beta/2}(1+s)^{\gamma/2}\},$ where $s>0,\beta+\gamma\in(0,2].$ Furthermore, $B_d(t)$ and $L_{\beta,\gamma}(t)$ are assumed independent.
In \cite{anh2017stochastic}, the authors considered the following Cauchy problem

\begin{align}\label{beri}
&D_t^\alpha u = -\lambda (-\Delta)^{\beta/2}(I-\Delta)^{\gamma/2}u, \quad, \gamma\geq 0, \alpha \in (0,1], \beta\in (0,2],\\
&u(x,0) = \delta(x),\label{beri1}
\end{align}
where the time-fractional derivative is in the sense of Caputo.

The Fourier transform of the solution for the Cauchy problem \eqref{beri}-\eqref{beri1} is given by 
\begin{equation}
\widehat{u}(t,\xi) = E_{\alpha,1}\left(-\lambda t^\alpha||\xi||^\beta(1+||\xi||^2)^{\gamma/2}\right).
\end{equation}
An interesting stochastic interpretation of this result based on the Bessel-Riesz distribution has been provided in \cite{anh2017stochastic}.

Here we consider some new applications of Bessel-Riesz operators to generalized telegraph-type equations. The aim of this section is to consider analytical and probabilistic results about the generalized telegraph-type equation
\begin{align}\label{te}
(D_t^{\alpha})^2 u +2\lambda D_t^{\alpha} u = -(-\Delta)^{\beta/2}(I-\Delta)^{\gamma/2}u, \quad, \gamma\geq 0, \beta \in (0,2], \alpha\in (0,1].
\end{align}
As an application of the general Theorem \ref{thm:ft} we have the following result.

\begin{prop}
The solution of the fractional Cauchy problem involving the Bessel-Riesz operator
\begin{align}
&\bigg[(D_t^{\alpha})^2+2\lambda D_t^{\alpha}+(-\Delta)^{\beta/2}(I-\Delta)^{\gamma/2}\bigg] u =0,\quad t> 0, \ \alpha\in (0,1], \beta \in(0,2], \gamma \geq 0, \label{es2eq:fractelc}\\
&u(0)=f,\ \alpha \in(0,1],\label{es2eq:fractelc2}\\
& D_t^{\alpha}u(0)=0,\ \alpha \in(1/2,1],\label{es2eq:fractelc3}
\end{align}
is given by the convolution product $u_\alpha(t,x)=(f\ast \Phi_\alpha(t))(x),$ for each $x\in\mathbb R^d,$where
\begin{equation}
\Phi_\alpha(t,x)=\frac{1}{(2\pi)^{\frac d2}||x||^{\frac d2-1}}\int_0^\infty r^{\frac d2}J_{\frac d2-1}(r||x||)\widehat\Phi_\alpha(t,r)\de r,\quad t>0
\end{equation}
and 
\begin{align*}
\widehat \Phi_\alpha(t,\xi) &=\frac{1}{2}\left[
\bigg(1+\frac{\lambda}{\sqrt{\lambda-||\xi||^\beta(1+||\xi||^2)^{\gamma/2}}}\bigg)E_{\alpha,1}(r_1t^{\alpha})\right.\\
&\quad\left.+\bigg(1-\frac{\lambda}{\sqrt{\lambda-||\xi||^\beta(1+||\xi||^2)^{\gamma/2}}}\bigg)E_{\alpha,1}(r_2t^{\alpha})\right],
\end{align*}
where 
\begin{align}
\nonumber & r_1 := -\lambda+\sqrt{\lambda^2-||\xi||^\beta(1+||\xi||^2)^{\gamma/2}},\\
\nonumber & r_2 := -\lambda-\sqrt{\lambda^2-||\xi||^\beta(1+||\xi||^2)^{\gamma/2}}.
\end{align}
Moreover, for $\alpha \in (0,\frac{1}{2}]$ and $\beta+\gamma\in(0,2],$ the solution to the problem \eqref{es2eq:fractelc}-\eqref{es2eq:fractelc2} has the following stochastic representation 
\begin{equation*}
u_\alpha(t,x)=\mathbb E_x\left[f(Y_d^{\beta,\gamma}(\mathcal{L}^\alpha(t) ))\right] ,\quad t>0,x\in\mathbb R^d. 
\end{equation*}
\end{prop}
\begin{proof}
 From the above discussion, the results contained in the theorem follows immediately from Theorem \ref{thm:ft} and Corollary \ref{cor1} by observing that $m(\xi)=||\xi||^\beta(1+||\xi||^2)^{\gamma/2}.$   
\end{proof}

\begin{os}
If we consider the particular case $\alpha =1$ in \eqref{es2eq:fractelc2} we get the Bessel-Riesz telegraph equation
\begin{equation}\label{te1}
\bigg[\partial_{t}^2+2\lambda\partial_t\bigg] u(t,x) =- (-\Delta)^{\beta/2}(I-\Delta)^{\gamma/2}u(t,x), \quad x \in \mathbb{R}^d.
\end{equation}
We obtain a simple solution for the Cauchy problem \eqref{es2eq:fractelc}-\eqref{es2eq:fractelc3} since $\widehat \Phi_1(t,\xi)$ can be represented by using exponential functions
\begin{align*}
\widehat \Phi_1(t,\xi)& = \frac{e^{-\lambda t}}{2}\bigg[\left(1+\frac{\lambda}{\sqrt{\lambda^2+||\xi||^\beta(1+||\xi||^2)^{\gamma/2}}}\right)e^{t\sqrt{\lambda^2-||\xi||^\beta(1+||\xi||^2)^{\gamma/2}}}+\\
\nonumber &+\left(1-\frac{\lambda}{\sqrt{\lambda^2+||\xi||^\beta(1+||\xi||^2)^{\gamma/2}}}\right)e^{-t\sqrt{\lambda^2-||\xi||^\beta(1+||\xi||^2)^{\gamma/2}}}\bigg]\\
&=e^{-\lambda t}\left[\cosh(t\sqrt{\lambda^2-||\xi||^\beta(1+||\xi||^2)^{\gamma/2}}+\frac{\lambda}{\sqrt{\lambda-m(\xi)}}\sinh(t\sqrt{\lambda^2-||\xi||^\beta(1+||\xi||^2)^{\gamma/2}}\right]
\end{align*}
\end{os}

\subsection{The time-fractional telegraph-type relativistic diffusion}

Following the recent paper by Shieh \cite{shieh2012time}, we recall that the time-fractional relativistic diffusion equation is the equation
\begin{equation}\label{rel}
\partial_t u = H_{\nu,m} u,\quad \alpha\in(0,1],
\end{equation}
involving the operator
\begin{equation}
H_{\nu,m}:= m-(m^{\frac{2}{\nu}}-\Delta)^{\frac{\nu}{2}}, 
\end{equation}
that is the relativistic diffusion operator with the spatial-fractional parameter $\nu \in (0,2)$ and the normalized mass parameter $m>0$. We have that the relativistic operator admits Fourier multiplier $[m-(m^{\frac{2}{\nu}}+||\xi||^2)^{\frac{\nu}{2}}]$; that is 
$$(\mathcal FH_{\nu,m}  f)(\xi)=[m-(m^{\frac{2}{\nu}}+||\xi||^2)^{\frac{\nu}{2}}]\hat f(\xi),$$
where Dom$(H_{\nu,m} )=\{f\in\ L^2(\mathbb R^d,\de x):\int_{\mathbb R^d}[m-(m^{\frac{2}{\nu}}+||\xi||^2)^{\frac{\nu}{2}}]|\hat f(\xi)|^2\de \xi<\infty \}.$
Considering the Equation \eqref{rel} with the initial condition $u(x,0) = f(x)$, we recall that the probabilistic meaning of the solution for the Cauchy problem \eqref{rel} is given by (see \cite{shieh2012time}, Proposition 3)
\begin{equation}\label{rel1}
    u(t,x)  = \mathbb{E}_x\left[f(W_d^\nu(t))\right],
\end{equation}
where $W_d^\nu(t),t\geq 0,$ is the Brownian motion on $\mathbb{R}^d$ subordinated by a $\nu/2$ relativistic subordinator. 
The relativistic subordinator $T(t),t\geq  0,$ was introduced by Ryznar in \cite{ryznar2002estimates} as a Lévy process with increasing sample paths with Laplace function given by 
$$\mathbb E\left[e^{- uT(t)}\right]=e^{-t( (m^{\frac{2}{\nu}}-u)^{\frac{\nu}{2}}-m)},\quad u>0.$$ Assuming that the relativistic subordinator and the Brownian motion are independent, the subordinated process $W_d^\nu$ comparing in \eqref{rel1} is a Lévy process on $\mathbb{R}^d$ with characteristic function given by
$$\mathbb E\left[e^{i\langle \xi, W_d^\nu(t)\rangle}\right]=e^{-t ((m^{\frac{2}{\nu}}+||\xi||^2)^{\frac{\nu}{2}}-m)},\quad \xi\in\mathbb R^d.$$

Furthermore, in \cite{shieh2012time} the time-fractional generalization of the equation \eqref{rel} based on the replacement of the classical time-derivative with the Caputo derivative has been studied, while in \cite{beghin2019note} the authors dealt with a time-fractional generalized version of the equation \eqref{rel}.  Here we consider the fractional telegraph-type relativistic diffusion equation. 


\begin{prop}
The solution of the fractional telegraph-type relativistic diffusion 
\begin{align}
&\bigg[(D_t^{\alpha})^2+2\lambda D_t^{\alpha}-H_{\nu,m}\bigg] u =0,\quad t> 0, \ \alpha, \nu \in(0,1], \label{es3eq:fractelcrel}\\
&u(0)=f,\ \alpha \in(0,1],\label{es3eq:fractelc2rel}\\
& D_t^{\alpha}u(0)=0,\ \alpha \in(1/2,1],\label{es3eq:fractelc3rel}
\end{align}
is given by the convolution product $u(t,x)=(f\ast \Phi_\alpha(t))(x),$ for each $x\in\mathbb R^d,$where
\begin{equation*}
\Phi_\alpha(t,x)=\frac{1}{(2\pi)^{\frac d2}||x||^{\frac d2-1}}\int_0^\infty r^{\frac d2}J_{\frac d2-1}(r||x||)\widehat\Phi_\alpha(t,r)\de r,\quad t>0
\end{equation*}
and 
\begin{align*}
    \widehat\Phi_\alpha(t,\xi) = &= \frac{1}{2}\bigg[
\bigg(1+\frac{\lambda}{\sqrt{\lambda-\theta(||\xi||)}}\bigg)E_{\alpha,1}(r_1t^\alpha)+\\
&+\bigg(1-\frac{\lambda}{\sqrt{\lambda-\theta(||\xi||)}}\bigg)E_{\alpha,1}(r_2t^\alpha)\bigg],
\end{align*}
where 
$$\theta (||\xi||) = (m^{\frac{2}{\nu}}+||\xi||^2)^{\frac{\nu}{2}}-m$$
and
\begin{align}
\nonumber & r_1 = -\lambda+\sqrt{\lambda^2-\theta (||\xi||)},\\
\nonumber & r_2 = -\lambda-\sqrt{\lambda^2-\theta (||\xi||)}.
\end{align}
Moreover, for $\alpha \in (0,\frac{1}{2}],$ the solution to the problem \eqref{es3eq:fractelcrel}-\eqref{es3eq:fractelc2rel} has the following stochastic representation 
\begin{equation}
u_\alpha(t,x)=\mathbb E_x\left[f(W_d^\nu(\mathcal{L}^\alpha(t) ))\right] ,\quad t>0,x\in\mathbb R^d. 
\end{equation}
\end{prop}
 \begin{proof} As in the previous sections, the statements of the theorem follows from Theorem \ref{thm:ft} and Corollary \ref{cor1} by observing that $m(\xi)=\theta (||\xi||).$  
 \end{proof}

\section{The Abstract Euler-Poisson-Darboux equation}

\subsection{Stochastic solution to the Euler-Poisson-Darboux equation} The Euler-Poisson-Darboux (EPD) equation is an hyperbolic telegraph-type equation with singular coefficient whose classical form is given by
\begin{equation}\label{eps}
    \bigg[\partial_{t}^2+\frac{2\lambda}{t}\partial_t\bigg] u(t,x) =\Delta u(t,x),\quad t>0,x\in\mathbb R^d.
\end{equation}
There is a wide literature about this equation, that is still object of recent research, we refer for example to the recent paper \cite{shishkina2017general} and the references therein. Random motions related to \eqref{eps} have been studied in \cite{garra2014random}. 
As pointed out in the literature, the solution for the EPD equation \eqref{eps} can be obtained by means of the fractional integral of the d'Alembert solution to the wave equation with the same initial data. In this section we consider the abstract EPD, showing the connection with the solution of the abstract wave equation. Then, we will get the stochastic interpretation in terms of the D'Alembert solution of the wave equation.

Let $A$ be a positive and injective self-adjoint operator on a Hilbert space $\mathbb H$ as in Section \ref{sec:intrfte}. In particular the same spectral decomposition of $A$ holds true; i.e. there exists a unitary group $U$ and positive measurable function $m$ such that \eqref{eq:specdec} fulfills.
Let us deal with the following abstract  Euler-Poisson-Darboux equation

\begin{align}\label{abepd}
&\bigg[\partial_{t}^2+\frac{2\lambda}{t}\partial_t+A\bigg] u =0,\quad t> 0,\lambda>0,\\
&u(0)=f,\quad \partial_t u(0)=0,\label{abepd1}
\end{align}
where $t\mapsto u(t)\in \mathbb H$ and $f\in D(A).$ The unique solution of the abstract Cauchy problem \eqref{abepd}-\eqref{abepd1} is given by $u(t)= U^{-1} \widehat u (t)$ where $\widehat u(t,\xi)\in L^2(\Omega, \Sigma, \mu).$
The main result of this section is contained in the next theorem.
\begin{theorem}
The function $\widehat{u}$ has the following representations
\begin{align}
\label{lebe}\widehat{u}(t,\xi) &= \left(\frac{2}{t \sqrt{m(\xi)}}\right)^{\lambda-\frac{1}{2}}\Gamma\left(\lambda+\frac{1}{2}\right)J_{\lambda-\frac{1}{2}}\left(t\sqrt{m(\xi)}\right)\widehat f(\xi)\\
&= \frac{\widehat f(\xi)}{B(\lambda,\frac12)t}\int_{-t}^{t}\left(1-\frac{w^2}{t^2}\right)^{\lambda-1}\left(\frac{e^{i \sqrt{m(\xi)}w}+e^{-i\sqrt{m(\xi)}w}}{2}\right)\de w\nonumber \\
&= \frac{1}{B(\lambda,\frac12)t}\int_{-t}^{t}\left(1-\frac{w^2}{t^2}\right)^{\lambda-1}\widehat{\bar{u}}(\xi,w)\de w, \nonumber
\end{align} 
where $\bar{u}$
is the unique solution of the abstract wave equation $$\left(\partial_{t}^2+A\right)u = 0 $$
with the  initial conditions \eqref{abepd1}. Furthermore, the unique solution to the Cauchy problem \eqref{abepd} is given by 
\begin{equation}\label{eq:stocsolepd}
 u(t)=\mathbb E[\bar u(t\sqrt \mathfrak X)],   
\end{equation}
where $\mathfrak X$ is a Beta r.v with parameters $\frac12$ and $\lambda.$
\end{theorem}
\begin{proof}
Since $A$ is self-adjoint on $\mathbb H$ and positive, there exists a unique positive self-adjoint operator $\sqrt A$ such that $(\sqrt A)^2=A.$ First of all, let us consider the abstract wave equation
    \begin{equation}\label{naw}
        \left(\partial_{t}^2+A\right) \bar u(t) = \left (\partial_{t}+i \sqrt{A}\right)\left(\partial_{t}-i\sqrt{A}\right)\bar u(t)=0,
    \end{equation}
    where $$\bar u(t)=\frac{u_1(t)+u_2(t)}{2}.$$
    Let be $u_k(t)= U^{-1}\widehat  u_k(t), k=1,2,$ be the unique solution of
\begin{equation}\label{eq:schrod}
\begin{cases}
 \partial_t u_k=(-1)^{k+1} i\sqrt{A} u_k,\\
   u_k(0)=f.
   \end{cases}
\end{equation} 
and for the spectral theorem
$\widehat  u_k(\xi ,t),=e^{ (-1)^{k+1} i \sqrt {m(\xi}) t}\widehat f(\xi), k=1,2,$ is the unique solution of 
\begin{equation}\label{eq:schrod}
\begin{cases}
 \partial_t \widehat u_k(t,\xi)=(-1)^{k+1} i\sqrt{m(\xi)} \widehat u_k(t,\xi),\\
  \widehat u_k(\xi,0)=\widehat f(\xi),
   \end{cases}
\end{equation} 
Under our assumptions the problem \eqref{eq:schrod}-\eqref{abepd1} is well-posed (see, e.g, Theorem 7.4 in \cite{goldstein2017semigroups}), and then we can say that  $\bar u$ is the unique solution to the abstract wave equation \eqref{naw}.\\
Going back to the abstract EPD equation, by using again the spectral theorem for the Cauchy problem we have that the self-adjoint operator can be treated as a scalar constant and then $u(t)= U^{-1} \widehat u (t),$ where $\widehat u(t,\xi)\in L^2(\Omega, \Sigma, \mu)$ satisfies

 \begin{align}\label{eq:unitaryoper}
&\partial_{t}^2\widehat u(t,\xi)+\frac{2\lambda}{t}\partial_t\widehat u(t,\xi)+m(\xi)\widehat u(t,\xi)  =0,\quad t>0,\\
&\widehat u(0,\xi)=\widehat f,\quad \partial_t \widehat  u(0,\xi)=0.
\end{align}
By recalling the solution of the Bessel equation we have that  (see, e.g., \cite{lebedev:72:special})
\begin{align*}
&\widehat{u}(t,\xi) = \left(\frac{2}{t\sqrt{m(\xi)}}\right)^{\lambda-\frac{1}{2}}\Gamma\left(\lambda+\frac{1}{2}\right)J_{\lambda-\frac{1}{2}}\left(t\sqrt{m(\xi)}\right)\widehat f(\xi)
\end{align*}
The Poisson integral representation of the Bessel functions reads
\begin{equation}\label{eq:poisson}
J_\mu(z)=\frac{(z/2)^\mu}{\sqrt\pi \Gamma(\mu+\frac12)}\int_{-1}^{+1}(1-w^2)^{\mu-\frac12}\cos(z w)\de w
\end{equation}
valid for $\mu>-\frac12,z\in\mathbb R$ (see \cite{lebedev:72:special}, pag. 114, formula (5.10.3)).
By inserting \eqref{eq:poisson} into \eqref{lebe}, we readily have that
\begin{align*}
\widehat u(t,\xi)&=\frac{\Gamma(\lambda+\frac{1}{2})}{\sqrt{\pi}\Gamma(\lambda)}\int_{-1}^1(1-w^2)^{\lambda-1}\cos(t\sqrt{ m(\xi)}w)\widehat f(\xi)\de w\\
&=\frac{\Gamma(\lambda+\frac{1}{2})}{\sqrt{\pi}\Gamma(\lambda)}\int_{-1}^1(1-w^2)^{\lambda-1}\left(\frac{e^{it \sqrt{m(\xi)}w}+e^{-it\sqrt{m(\xi)}w}}{2}\right)\widehat f(\xi)\de w\\
&=\frac{\Gamma(\lambda+\frac{1}{2})}{\sqrt{\pi}\Gamma(\lambda)t}\int_{-t}^{t}\left(1-\frac{w^2}{t^2}\right)^{\lambda-1}\left(\frac{e^{i \sqrt{m(\xi)}w}+e^{-i\sqrt{m(\xi)}w}}{2}\right)\widehat f(\xi)\de w\\
&=  \frac{\Gamma(\lambda+\frac{1}{2})}{\sqrt{\pi}\Gamma(\lambda)t}\int_{-t}^{t}\left(1-\frac{w^2}{t^2}\right)^{\lambda-1}\widehat{\bar{u}}(\xi,w)\de w \nonumber
\end{align*}
as claimed. Furthermore, it is not hard to prove that
$$\widehat u(t,\xi)=\mathbb E[\widehat{\bar{u}}(t\sqrt{\mathfrak X},\xi)],$$
where $\mathfrak X$ is a Beta r.v with parameters $\frac12$ and $\lambda,$ and then the representation \eqref{eq:stocsolepd} immediately follows. 
\end{proof}

\begin{remark}
    We recall that it is possible to get the solution $\bar u(t)$ to the abstract wave equation \eqref{naw}, by means of the spectral functional calculus as follows
    $$\bar u(t)=\cos(t\sqrt A)f=\int_0^\infty \cos(t\sqrt x)\de{\bf E}(x) f,$$
    where ${\bf E}$ represents the resolution of the identity for $\sqrt A$ (see, e.g., \cite{goldstein2017semigroups}). When $A$ admits discrete spectrum, the above spectral integral reduces to the following Fourier series
    $$\bar u(t)=\sum_{k=1}^\infty \langle f,\phi_k\rangle \cos(t\sqrt{x_k})\phi_k,$$
    where $\{x_k:k\geq 1\}$ is a sequence of positive eigenvalues and $\{\phi_k:k\geq 1\}$ is an orthonormal basis of eigenvectors such that $A\phi_k=x_k\phi_k.$
\end{remark}

\begin{remark}
Theorem 4 proves that the solution of the Cauchy problem \eqref{abepd}-\eqref{abepd1} for the abstract EPD equation can be represented as the Erd\'elyi-Kober fractional integral of the solution of the abstract wave equation \eqref{naw}. We refer directly to the paper written by Erd\'elyi \cite{erd1970} that gives the first proof of the relation between the solution of the EPD equation and the D'Alembert solution, i.e. in the particular case $A = -\Delta $. This relation has been object of many studies. We refer for example to the paper by Rosencrans \cite{rosencrans1973diffusion}, where the author obtained a representation similar to \eqref{eq:stocsolepd}.

We recall that the Erd\'elyi-Kober fractional integral is defined as (see, e.g., \cite{kilbas2006})
$$(I_\alpha^{m}f)(x)=\frac{m}{\Gamma(\alpha)}\int_0^x (x^m-y^m)^{\alpha-1}y^{m-1}f(y)\de y,\quad \alpha>0,m>0.$$
Therefore, we can observe that \eqref{eq:stocsolepd} can be expressed in terms of the Erd\'elyi-Kober fractional integral as follows
$$u(t)=\frac{\Gamma(\lambda+\frac12)}{\sqrt{\pi}}I_\lambda^1(\bar u(t\sqrt{y})/\sqrt y).$$
\end{remark}
\begin{remark}
 Let $$\widehat{u}_0(t,\xi):=\left(\frac{2}{t \sqrt{m(\xi)}}\right)^{\lambda-\frac{1}{2}}\Gamma\left(\lambda+\frac{1}{2}\right)J_{\lambda-\frac{1}{2}}\left(t\sqrt{m(\xi)}\right).$$ If $U=\mathcal F,$ with $\mathcal F:\mathbb H\to L^2(\mathbb R^d,\de x),$ $U^{-1}=\mathcal F^{-1},$ and $m(\xi)=m(||\xi||),$  for $t\geq0$ and $x\in\mathbb R^d,$ we have the following alternative representation of the solution to the EPD equation $$u(t,x)=(f* \Psi(t))(x),$$
 where
 \begin{align*}
   \Psi(t,x)&=(\mathcal F^{-1}\widehat{u}_0(t))(x)\\
   &=\frac{1}{(2\pi)^d}\int_{\mathbb R^d} e^{-i\langle x,\xi\rangle}\widehat{u}_0(t,||\xi||) \de \xi\\
   &=\frac{\Gamma(\lambda +\frac12)}{(2\pi)^{\frac d2}||x||^{\frac d2-1}}\left(\frac2t\right)^{\lambda-\frac12}\int_0^\infty r^{\frac d2}J_{\lambda-\frac12}(t\sqrt{m(r)})J_{\frac d2-1}(r||x||)   \de r,
   \end{align*}
   where in the last step we have used the same approach developed in the proof of Theorem \ref{thm:ft}.
 
\end{remark}

Finally, we suggest a heuristic construction of the solution to the EPD equation by means of a finite velocity random process. 
Let $X_\varepsilon:=\{X_\varepsilon(t), t\geq 0\}$ a telegraph process (starting its motion upward) defined as follows
\begin{equation*}
X_\varepsilon(t)=\int_0^t (-1)^{N_\varepsilon(s)}\de s,
\end{equation*}
where $\{N_\varepsilon(t), t\geq 0\}$ is an inhomogenous Poisson process with rate function $\lambda (t)=\lambda/(t+\varepsilon),$ for $\varepsilon>0,$ and then $\Lambda(t)=\int_0^t\lambda(s)\de s=\lambda[\log(t+\varepsilon)-\log(\varepsilon)]<\infty.$ 
Therefore $ u_\varepsilon(t)=\mathbb E[\bar u(X_\varepsilon(t))]$ satisfies (see \cite{kaplan1964differential})
\begin{equation}\label{eq:pdeitp}
\bigg[\partial_{t}^2+\frac{2\lambda}{t+\varepsilon}\partial_t\bigg] u(t) = -Au(t).
\end{equation}
Therefore \eqref{eq:pdeitp} suggests that $u_\varepsilon(t)$ tends to the solution \eqref{eq:stocsolepd} of the EPD equation, as $\varepsilon\to0$; that is
$$ u_\varepsilon(t)\to u_0(t)=\bar u(t),\quad \varepsilon\to0.$$

\subsection{Applications: The space-fractional EPD equations}
The space-fractional Euler-Poisson-Darboux equation, where $A=-(-\Delta)^{\beta/2}$ defined as in Section \ref{sec:fraclaplace}, has been recently studied in \cite{de2020random}.
On the basis of the previous analysis we have a complete picture about the solution of a space-fractional EPD equation. 
Let us consider the following Cauchy problem 
 \begin{align}\label{fepd}
    \begin{cases}
\left[\partial_{t}^2+\frac{2\lambda}{t}\partial_t-(-\Delta)^{\beta/2}\right] u(t,x)=0, \quad \beta \in (0,2],t>0,x\in\mathbb R^d,\\
u(0,x)=f(x),\\
\partial_t u(0,x)=0,
\end{cases}
\end{align}
with $f\in$ Dom$((-\Delta)^{\beta/2}).$
We have the following result.
\begin{prop}
The Fourier transform of the solution of the problem \eqref{fepd} is given by 
\begin{align*}
\label{lebe}\widehat{u}(t,\xi) &= \left(\frac{2}{t ||\xi||^{\beta/2}}\right)^{\lambda-\frac{1}{2}}\Gamma\left(\lambda+\frac{1}{2}\right)J_{\lambda-\frac{1}{2}}\left(t||\xi||^{{\beta}/2}\right)\widehat f(\xi)\\
&= \frac{1}{B(\lambda,\frac12)t}\int_{-t}^{t}\left(1-\frac{w^2}{t^2}\right)^{\lambda-1}\left(\frac{e^{i ||\xi||^{{\beta}/2}w}+e^{-i||\xi||^{{\beta}/2}w}}{2}\right)\widehat f(\xi)\de w. 
\end{align*} 
Furthermore, we have the following stochastic solution
\begin{align*}
    u(t,x)&=  \mathbb E[\widebar{u}(t\sqrt{\mathfrak X},x)] 
 \end{align*}
where $\widebar{u}(t,x)=(\mathcal F^{-1}\cos(t||\xi||^{\beta/2})\widehat{f}(\xi))(x)$ is the solution to the abstract wave equation \eqref{naw}.
\end{prop}
\begin{proof}
In fact, this is a corollary of the Theorem 4, by using the Fourier transform of the fractional Laplacian of order $\beta/2$ and $m(\xi)=||\xi||^\beta$. 
\end{proof}

\section{Conclusions}
In this paper, we have studied the abstract time-fractional telegraph equation from both analytical and probabilistic viewpoints. We provide a general scheme for obtaining the exact solutions to applied problems, such as those involving the fractional Laplacian or the Bessel-Riesz operator. The second part of the paper considers the abstract Euler-Poisson-Darboux (EPD) equation. The connection to the first part is twofold: First, the EPD equation can be viewed as a telegraph equation with a variable (singular) rate. Second, it is well-known that the EPD solution is related to the D'Alembert solution of the classical wave equation via a fractional integral. In this work, we consider the abstract EPD and provide the analytical and probabilistic representation of its solution. We also show that the general theory can be applied to solve interesting mathematical problems, such as the space-fractional EPD equation.

The abstract fractional equations considered here involve time-fractional derivatives in the sense of Caputo. An interesting problem for future research is to study the general theory for abstract telegraph-type equations that involve integro-differential operators with different memory kernels. This research direction would follow recent work on anomalous diffusion involving general fractional derivatives with Sonin Kernels (see, e.g., \cite{kochubei2011general}, \cite{toaldo2015convolution}, \cite{chen2017time}, \cite{gorska2020generalized}, \cite{alkandari2025anomalous} and the references therein).

Another open question is to determine which fractional version of the EPD equation is most useful to consider. The approach based purely on time-fractional derivatives does not seem to be effective in this particular case. Indeed, since in  the operator $(D_t^\alpha)^2+\frac{2\lambda}{t} D_t^\alpha$ the term $\frac{2\lambda}{t}$ appears,  the Laplace transform method is challenge. A possible alternative is to deal with fractional powers of the time operator that appears in the EPD equation.

\section*{acknowledgments}
 The research of ADG is partially supported by  Italian Ministry of University and Research (MUR) under PRIN 2022 (APRIDACAS), Anomalous Phenomena on Regular and Irregular Domains: Approximating Complexity for the Applied Sciences, Funded by EU - Next Generation EU
CUP B53D23009540006 - Grant Code 2022XZSAFN - PNRR M4.C2.1.1.

	\bibliographystyle{abbrv}
	\bibliography{references}

\begin{thebibliography}{10}

\bibitem{alkandari2025anomalous}
M.~Alkandari, D.~Loutchko, and Y.~Luchko.
\newblock Anomalous diffusion models involving regularized general fractional derivatives with sonin kernels.
\newblock {\em Fractal and Fractional}, 9(6):363, 2025.

\bibitem{angelani2024anomalous}
L.~Angelani, A.~De~Gregorio, R.~Garra, and F.~Iafrate.
\newblock Anomalous random flights and time-fractional run-and-tumble equations.
\newblock {\em Journal of Statistical Physics}, 191(10):129, 2024.

\bibitem{anh2017stochastic}
V.~Anh, N.~Leonenko, and A.~Sikorskii.
\newblock Stochastic representation of fractional {B}essel-{R}iesz motion.
\newblock {\em Chaos, Solitons \& Fractals}, 102:135--139, 2017.

\bibitem{anh2004riesz}
V.~Anh and R.~McVinish.
\newblock The {R}iesz-{B}essel fractional diffusion equation.
\newblock {\em Applied Mathematics and Optimization}, 49(3):241--264, 2004.

\bibitem{ashurov2023fractional}
R.~Ashurov and R.~Saparbayev.
\newblock Fractional telegraph equation with the {C}aputo derivative.
\newblock {\em Fractal and Fractional}, 7(6):483, 2023.

\bibitem{beghin2019note}
L.~Beghin and R.~Garra.
\newblock A note on the generalized relativistic diffusion equation.
\newblock {\em Mathematics}, 7(11):1009, 2019.

\bibitem{bresters}
D.~W. Bresters.
\newblock On a generalized {E}uler–{P}oisson–{D}arboux equation.
\newblock {\em SIAM Journal on Mathematical Analysis}, 9(5):924--934, 1978.

\bibitem{cascaval2002fractional}
R.~C. Cascaval, E.~C. Eckstein, C.~L. Frota, and J.~A. Goldstein.
\newblock Fractional telegraph equations.
\newblock {\em Journal of Mathematical Analysis and Applications}, 276(1):145--159, 2002.

\bibitem{chen2008analytical}
J.~Chen, F.~Liu, and V.~Anh.
\newblock Analytical solution for the time-fractional telegraph equation by the method of separating variables.
\newblock {\em Journal of Mathematical Analysis and Applications}, 338(2):1364--1377, 2008.

\bibitem{chen2017time}
Z.-Q. Chen.
\newblock Time fractional equations and probabilistic representation.
\newblock {\em Chaos, Solitons \& Fractals}, 102:168--174, 2017.

\bibitem{de2020random}
A.~De~Gregorio and E.~Orsingher.
\newblock Random flights connecting porous medium and {E}uler--{P}oisson--{D}arboux equations.
\newblock {\em Journal of Mathematical Physics}, 61(4), 2020.

\bibitem{d2014time}
M.~D’Ovidio, E.~Orsingher, and B.~Toaldo.
\newblock Time-changed processes governed by space-time fractional telegraph equations.
\newblock {\em Stochastic Analysis and Applications}, 32(6):1009--1045, 2014.

\bibitem{eckstein1999mathematics}
E.~C. Eckstein, J.~A. Goldstein, and M.~Leggas.
\newblock The mathematics of suspensions: {K}ac walks and asymptotic analyticity.
\newblock In {\em Proceedings of the Fourth Mississippi State Conference on Difference Equations and Computational Simulations}, volume~3, pages 39--50, 1999.

\bibitem{erd1970}
A.~Erdelyi.
\newblock On the euler-poisson-darboux equation.
\newblock {\em Journal d'Analyse Mathematique}, 23(1):89--102, 1970.

\bibitem{garra2014random}
R.~Garra and E.~Orsingher.
\newblock Random flights related to the {E}uler-{P}oisson-{D}arboux equation.
\newblock {\em Markov Processed \& Related Fields}, 22:87--110, 2014.

\bibitem{goldstein2017semigroups}
J.~A. Goldstein.
\newblock {\em Semigroups of linear operators and applications}.
\newblock Courier Dover Publications, 2017.

\bibitem{gorenflo2020mittag}
R.~Gorenflo, A.~A. Kilbas, F.~Mainardi, and S.~Rogosin.
\newblock {\em Mittag-Leffler Functions, Related Topics and Applications}.
\newblock Springer Nature, 2020.

\bibitem{gorska2020generalized}
K.~G{\'o}rska, A.~Horzela, E.~Lenzi, G.~Pagnini, and T.~Sandev.
\newblock Generalized cattaneo (telegrapher's) equations in modeling anomalous diffusion phenomena.
\newblock {\em Physical Review E}, 102(2):022128, 2020.

\bibitem{gradshteyn2014table}
I.~S. Gradshteyn and I.~M. Ryzhik.
\newblock {\em Table of integrals, series, and products}.
\newblock Academic press, 2014.

\bibitem{griego1971}
R.~J. Griego and R.~Hersh.
\newblock Theory of random evolutions with applications to partial differential equations.
\newblock {\em Transactions of the American Mathematical Society}, 156:405--418, 1971.

\bibitem{kaplan1964differential}
S.~Kaplan.
\newblock Differential equations in which the poisson process plays a role.
\newblock {\em Bull. Amer. Math. Soc.}, 70(6):264--268, 1964.

\bibitem{kilbas2006}
A.~A. Kilbas, H.~M. Srivastava, and J.~J. Trujillo.
\newblock {\em Theory and Applications of Fractional Differential Equations, Volume 204 (North-Holland Mathematics Studies)}.
\newblock Elsevier Science Inc., USA, 2006.

\bibitem{kilbas2006theory}
A.~A. Kilbas, H.~M. Srivastava, and J.~J. Trujillo.
\newblock {\em Theory and Applications of Fractional Differential Equations, Volume 204 (North-Holland Mathematics Studies)}.
\newblock Elsevier Science Inc., USA, 2006.

\bibitem{kochubei2011general}
A.~N. Kochubei.
\newblock General fractional calculus, evolution equations, and renewal processes.
\newblock {\em Integral Equations and Operator Theory}, 71(4):583--600, 2011.

\bibitem{kwasnicki2017ten}
M.~Kwa{\'s}nicki.
\newblock Ten equivalent definitions of the fractional {L}aplace operator.
\newblock {\em Fractional Calculus and Applied Analysis}, 20(1):7--51, 2017.

\bibitem{lebedev:72:special}
N.~N. Lebedev.
\newblock {\em Special functions and their applications}.
\newblock Dover Publications, Inc., New York, 1972.
\newblock Revised edition, translated from the Russian and edited by Richard A. Silverman, Unabridged and corrected republication.

\bibitem{li2019fractional}
C.-G. Li, M.~Li, S.~Piskarev, and M.~M. Meerschaert.
\newblock The fractional d'alembert's formulas.
\newblock {\em Journal of Functional Analysis}, 277(12):108279, 2019.

\bibitem{martinucci2025finite}
B.~Martinucci and S.~Spina.
\newblock On a finite-velocity random motion governed by a modified {E}uler-{P}oisson-{D}arboux equation.
\newblock {\em Electronic Journal of Probability}, 30:1--28, 2025.

\bibitem{masoliver1996finite}
J.~Masoliver and G.~H. Weiss.
\newblock Finite-velocity diffusion.
\newblock {\em European Journal of Physics}, 17(4):190, 1996.

\bibitem{meerschaert2019inverse}
M.~M. Meerschaert, E.~Nane, and P.~Vellaisamy.
\newblock Inverse subordinators and time fractional equations.
\newblock {\em Handbook of Fractional Calculus with Applications}, 1:9783110571622--017, 2019.

\bibitem{meerschaert2013inverse}
M.~M. Meerschaert and P.~Straka.
\newblock Inverse stable subordinators.
\newblock {\em Mathematical modelling of natural phenomena}, 8(2):1--16, 2013.

\bibitem{metzler2000random}
R.~Metzler and J.~Klafter.
\newblock The random walk's guide to anomalous diffusion: a fractional dynamics approach.
\newblock {\em Physics reports}, 339(1):1--77, 2000.

\bibitem{orsingher2004time}
E.~Orsingher and L.~Beghin.
\newblock Time-fractional telegraph equations and telegraph processes with brownian time.
\newblock {\em Probability Theory and Related Fields}, 128(1):141--160, 2004.

\bibitem{podlubny1998fractional}
I.~Podlubny.
\newblock {\em Fractional differential equations: an introduction to fractional derivatives, fractional differential equations, to methods of their solution and some of their applications}, volume 198.
\newblock Elsevier, 1998.

\bibitem{rosencrans1973diffusion}
S.~I. Rosencrans.
\newblock Diffusion transforms.
\newblock {\em Journal of Differential Equations}, 13(3):457--467, 1973.

\bibitem{ryznar2002estimates}
M.~Ryznar.
\newblock Estimates of green function for relativistic $\alpha$-stable process.
\newblock {\em Potential Analysis}, 17(1):1--23, 2002.

\bibitem{shieh2012time}
N.-R. Shieh.
\newblock On time-fractional relativistic diffusion equations.
\newblock {\em Journal of Pseudo-Differential Operators and Applications}, 3(2):229--237, 2012.

\bibitem{shishkina2017general}
E.~Shishkina and S.~Sitnik.
\newblock General form of the {E}uler-{P}oisson-{D}arboux equation and application of the transmutation method.
\newblock {\em Electronic Journal of Differential Equations}, 2017:177--177, 2017.

\bibitem{tawfik2018analytical}
A.~M. Tawfik, H.~Fichtner, R.~Schlickeiser, and A.~Elhanbaly.
\newblock Analytical solutions of the space--time fractional telegraph and advection--diffusion equations.
\newblock {\em Physica A: Statistical Mechanics and Its Applications}, 491:810--819, 2018.

\bibitem{toaldo2015convolution}
B.~Toaldo.
\newblock Convolution-type derivatives, hitting-times of subordinators and time-changed {$C_0$}-semigroups.
\newblock {\em Potential Analysis}, 42:115--140, 2015.

\end{thebibliography}

\end{document}